\documentclass[letterpaper,12pt]{article}

\usepackage[margin=0.5in]{geometry}

\usepackage{fullpage}
\usepackage{enumerate}
\usepackage{authblk}
\usepackage[colorinlistoftodos]{todonotes}
\usepackage{verbatim}
\usepackage{section, amsthm, textcase, setspace, amssymb, lineno, amsmath, amssymb, amsfonts, latexsym, fancyhdr, longtable, ulem}
\usepackage{epsfig, graphicx, pstricks,pst-grad,pst-text,tikz, tkz-berge,colortbl}
\usepackage{epsf}
\usepackage{graphicx, color}
\usepackage{float}
\usepackage[rflt]{floatflt}
\usepackage{amsfonts}
\usepackage{latexsym, enumitem}

\definecolor{darkgreen}{rgb}{0, 0.5, 0}

\theoremstyle{plain}

\newtheorem{theorem}{Theorem}
\newtheorem{lemma}[theorem]{Lemma}

\newtheorem{proposition}[theorem]{Proposition}

\theoremstyle{definition}
\newtheorem{definition}[theorem]{Definition}

\newtheorem{example}[theorem]{Example}
\newtheorem{conjecture}[theorem]{Conjecture}

\theoremstyle{remark}

\newcommand{\rn}[1]{{\color{red} #1}}

\newcommand{\bl}[1]{{\color{blue} #1}}

\newcommand{\ind}{{\rm ind }}

\newcommand{\p}{{\rm part }}

\newcount\tableauRow\newcount\tableauCol
\newcommand\Tableau[1]{%
  \begin{tikzpicture}[scale=0.5,draw/.append style={thick,black},baseline=4mm]
    \tableauRow=0
    \foreach \Row in {#1} {
       \tableauCol=1
       \foreach\k in \Row {
          \draw(\the\tableauCol,\the\tableauRow)+(-.5,-.5)rectangle++(.5,.5);
          \draw(\the\tableauCol,\the\tableauRow)node{\k};
          \global\advance\tableauCol by 1
       }
       \global\advance\tableauRow by -1
    }
  \end{tikzpicture}
}
\newcommand\Tabloid[1]{%
  \begin{tikzpicture}[scale=0.5,draw/.append style={thick,black},baseline=4mm]
    \tableauRow=0
    \foreach \Row in {#1} {
       \tableauCol=1
       \foreach\k in \Row {
          \draw($(\the\tableauCol,\the\tableauRow)+(-.5,-.5)$)--++(1,0);
          \draw($(\the\tableauCol,\the\tableauRow)+(-.5,.5)$)--++(1,0);
          \draw(\the\tableauCol,\the\tableauRow)node{\k};
          \global\advance\tableauCol by 1
       }
       \global\advance\tableauRow by -1
    }
  \end{tikzpicture}
}
\newcommand\ColumnTabloid[1]{%
  \begin{tikzpicture}[scale=0.5,draw/.append style={thick,black},baseline=4mm]
    \tableauRow=0
    \foreach \Row in {#1} {
       \tableauCol=1
       \foreach\k in \Row {
          \draw($(\the\tableauCol,\the\tableauRow)+(-.5,-.5)$)--++(0,1);
          \draw($(\the\tableauCol,\the\tableauRow)+(.5,-.5)$)--++(0,1);
          \draw(\the\tableauCol,\the\tableauRow)node{\k};
          \global\advance\tableauCol by 1
       }
       \global\advance\tableauRow by -1
    }
  \end{tikzpicture}
}

\title{Statistics on integer partitions arising from seaweed algebras}

\author[*]{Vincent E. Coll, Jr.}
\author[ ]{Andrew W. Mayers}
\author[*]{Nick W. Mayers}

\affil[*]{Department of Mathematics, Lehigh University, Bethlehem, PA, 18015}

\begin{document}
\maketitle


\noindent
\begin{abstract}
\noindent
Using the index theory of seaweed algebras, we explore various new integer partition statistics. We find relations to some well-known varieties of integer partitions as well as a surprising periodicity result.
\end{abstract}


\section{Introduction}

Partition statistics are often defined with an eye toward proving a congruence property.  An application of this principal can be found in the proofs of the famous congruence results of Ramanujan \textbf{\cite{raman}}, which were eventually established using the rank and crank statistics \textbf{\cite{crank,rankproof}}.  On the other hand, a partition statistic may present itself without prior appeal to an anticipated congruence property.  The recent index theory of seaweed algebras \textbf{\cite{Coll3, Coll1, Coll2, dk}} provides just such an instance.  Indeed, seaweed subalgebras of $\mathfrak{sl}(n)$ -- or simply, \textit{seaweeds} -- which are naturally defined in terms of two compositions of a single integer $n$, provide a sort of partition statistic generator.  One begins with a pair $(\lambda, \mu)$ of partitions of $n$. Since partitions are compositions, we can use $(\lambda, \mu)$ to define a seaweed subalgebra of $\mathfrak{sl}(n)$, whose index will be taken  to be the definition of the index of the partition pair $(\lambda, \mu)$. 

If we let $w(\lambda)$ denote the weight of $\lambda$, then
there are two choices of $\mu$ naturally associated with a given $\lambda$, namely, the trivial partition 
where $\mu= w(\lambda$), 
and it's conjugate $\mu^C=1^{w(\lambda)}$. Since the values of the index of $(\lambda, w(\lambda))$ and $(\lambda, 
1^{w(\lambda)})$ are reliant only on $\lambda$, the index of these partition pairs may be regarded as partition statistics on $\lambda$ alone.  Of course, the efficacy of such statistics must be adjudged according to their utility. However, we find that in each of these extremal cases, the index statistic connects to well-established  investigations. 

In the first case, $(\lambda, 1^{w(\lambda)})$, we find a connection to classical partition theory by establishing that the sequence $\{c^i_n\}_{n=1}^{\infty}$ defined by

$$
c^i_n=|\{\lambda\in\mathcal{P}(n) ~:~\ind_{1^{w(\lambda)}}(\lambda)=n-i\}|,
$$ for each fixed $i$ is eventually constant  -- converging to the number of partitions of $i-1$ into parts of two kinds. See Theorem~\ref{thm:ones}. 

In the second case, we consider seaweeds defined by a pair of compositions $(\lambda, w(\lambda))$. The enumeration of these composition pairs, when the index is zero, is of concern to Lie theorists.\footnote{Frobenius algebras form a distinguished class and have been extensively studied from the point of view of invariant theory \textbf{\cite{Ooms}} and are of special interest in deformation and quantum group theory resulting from their connection with the classical Yang-Baxter equation (see \textbf{\cite{G1}} and \textbf{\cite{G2}}).} Recent efforts to enumerate pairs of compositions that define a Frobenius (index zero) seaweed have concentrated on limiting the number of parts in the compositions. For example, Duflo (after the fashion of Coll et al \textbf{\cite{Collar}}), uses certain index-preserving operators on the set of compositions corresponding to a Frobenius seaweed subalgebra of $\mathfrak{sl}(n)$ to show that if $t$ is the number of parts in the defining compositions, then the 
number of such compositions is a rational polynomial of degree $\left[\frac{t}{2}\right]$ evaluated at $n$. See \textbf{\cite{df}}, Theorem 1.1 (b).  Dufflo's result is existential in nature.  However, if compositions are restricted to partitions and a modest limit is placed on the size of the parts -- rather than the number of parts -- the number of such compositions corresponding to a Frobenius seaweed subalgebra of $\mathfrak{sl}(n)$ becomes a periodic function of $n$. See Theorem~\ref{thm:period}.


\bigskip

The organization of the paper is as follows.  In Section 2 we develop the definitions and notation for integer partitions and seaweeds.  
In Section 3 we use the index theory of seaweeds to define the index of a partition and use this new definition to connect to some well-known classical investigations.  We conclude with some open questions.

\section{Preliminaries}\label{sec:prelim}
In Section 2.1 we review standard combinatorial notation.  In Section 2.2 we detail the recent index theory of seaweed algebras.  Throughout this article, we tacitly assume that all Lie algebras are over the complex numbers.

\subsection{Integer partitions}\label{Intger partitions}
We follow the notation of Andrews \textbf{\cite{andrews}} and adopt the following conventions.

\begin{definition}
A \textit{partition} $\lambda$ of a positive integer $n$ is a finite non-increasing sequence of positive integers $\lambda_1, \lambda_2, \dots, \lambda_m$ such that $n=\sum_{i=1}^{m}\lambda_i$.  The $\lambda_i$ are called the \textit{parts} of the partition and $w(\lambda)=n$ is the \textit{weight} of the partition.
\end{definition}

We will often employ the \textit{vector notation} for the partition
$\lambda=(\lambda_1, \lambda_2, \dots, \lambda_m)$. It will sometimes be useful to use a \textit{frequency notation} that makes explicit the number of times a particular integer occurs as a part of a partition.  So, if $\lambda=(\lambda_1, \lambda_2, \dots, \lambda_m)$, we alternatively write 

$$
\lambda = (1^{f_1}2^{f_2}3^{f_3}\cdots),
$$

\noindent
where exactly $f_i$ of the $\lambda_j$ are equal to $i$.

A graphical representation of a partition, called a \textit{Ferrers diagram}, is helpful to develop the notion of the conjugate of a partition.  More formally, the Ferrers diagram of a partition $\lambda=(\lambda_1,\hdots,\lambda_n)$ is a coordinatized  set of unit squares in the plane such that the lower left corner of each square will have integer coordinates $(i,j)$ such that $$0\ge i\ge -n+1,0\le j\le \lambda_{|i|+1}-1.$$ 
The Ferrers diagram of the partition $(4,2,1)$ is illustrated in the left-hand side of Figure 1.  The \textit{conjugate} of a partition $\mu$ is the partition $\mu^C$ resulting from exchanging the rows and columns in the Ferrers diagram associated to $\mu$.


\begin{example} The Ferrers diagram of the partition $\lambda= (4,2,1)$
and it's conjugate $\lambda^C= (3,2,1,1)$.

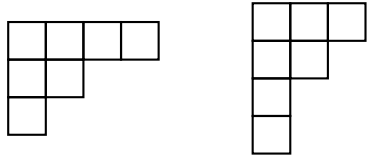
\begin{figure}[H]
$$\begin{tikzpicture}
	\node at (0,0) {\Tableau{{ , , , }, { , }, { }}};
    \node at (3, 0) {\Tableau{{ , , }, { , }, { }, { }}};
\end{tikzpicture}$$
\caption{Ferrers Diagram of $(4,2,1)$ and $(3,2,1,1)$}\label{fig:421YD}
\end{figure}

\end{example}

\subsection{Seaweed Algebras}\label{sect:swprelim}

In this section, we introduce seaweed algebras in type-A.\footnote{In \textbf{\cite{Panyushev1}}, Panyushev extended the Lie theoretic definition of seaweed algebras to the reductive case. If $\mathfrak{p}$ and $\mathfrak{p}$ are parabolic subalgebras of a reductive Lie algebra $\mathfrak{g}$ such that $\mathfrak{p}+\mathfrak{p'}=\mathfrak{g}$, then $\mathfrak{p}\cap\mathfrak{p'}$ is called a \textit{seaweed subalgebra o}f $\mathfrak{g}$ or simply $seaweed$ when $\mathfrak{g}$ is understood. For this reason, Joseph  has elsewhere 
 \textbf{\cite{Joseph}} called seaweed algebras, \textit{biparabolic}.  One can show that type-C and type-B seaweeds, in their standard representations, can be parametrized by a pair of partial compositions of $n$. See \textbf{\cite{CHM}}.}  These are seaweed subalgebras of $
\mathfrak{sl}(n)$ -- the set of all $n\times n$ matrices of trace zero. As we will see, such seaweed algebras are naturally defined in terms of two compositions of the positive integer $n$.  Recall that a \textit{composition} of $n$ is an unordered partition, which we will denote by $\lambda_1|\lambda_2|\cdots| \lambda_n$ to distinguish it from the ordered case in Definition 1, where there is an order relation on the $\lambda_i's$.

\begin{definition}
If $V$ is an $n$-dimensional vector space with a basis 
$\{e_1,\dots, e_n \}$, let $a_1|\dots|a_m$ and $b_1|\dots|b_l$ be two compositions of $n$ and consider the flags 

$$
\{0\} \subset V_1 \subset \cdots \subset V_{m-1} \subset V_m =V~~~\text{and}~~~ V=W_0\supset W_1\supset \cdots \supset W_t=\{0\}, 
$$
where $V_i=\text{span}\{e_1,\dots, e_{a_1+\cdots +a_i}\}$ and $W_j=\text{span}\{e_{b_1+\cdots +b_j+1},\dots, e_n\}$.  
\end{definition}

The subalgebra of $\mathfrak{sl}(n)$ preserving these flags is called a \textit{seaweed Lie algebra}, or simply \textit{seaweed}, and is denoted by the symbol 
$\displaystyle \frac{a_1|\cdots|a_m}{b_1|\cdots|b_t}$, which we  refer to as the \textit{type} of the seaweed.  If $b_1=n$, the seaweed is called \textit{maximal parabolic}.

\bigskip
\noindent
\textit{Remark:}  The preservation of flags in Definition 2 insures that seaweeds are closed under matrix multiplication, and therefore  define an associative algebra, hence also a Lie algebra under the commutator bracket.

\bigskip
  The evocative ``seaweed'' is descriptive of the shape of the algebra when exhibited in matrix form.   
For example, the seaweed algebra $\frac{2|4}{1|2|3}$ consists of traceless matrices of the form depicted on the left side of Figure \ref{fig:seaweed}, where * indicates the possible non-zero entries from the ground field, which we assume is the complex numbers.

\begin{figure}[H]
$$\begin{tikzpicture}[scale=0.75]
\draw (0,0) -- (0,6);
\draw (0,6) -- (6,6);
\draw (6,6) -- (6,0);
\draw (6,0) -- (0,0);
\draw [line width=3](0,6) -- (0,4);
\draw [line width=3](0,4) -- (2,4);
\draw [line width=3](2,4) -- (2,0);
\draw [line width=3](2,0) -- (6,0);

\draw [line width=3](0,6) -- (1,6);
\draw [line width=3](1,6) -- (1,5);
\draw [line width=3](1,5) -- (3,5);
\draw [line width=3](3,5) -- (3,3);
\draw [line width=3](3,3) -- (6,3);
\draw [line width=3](6,3) -- (6,0);

\draw [dotted] (0,6) -- (6,0);

\node at (.5,5.4) {{\LARGE *}};
\node at (.5,4.4) {{\LARGE *}};
\node at (1.5,4.4) {{\LARGE *}};
\node at (2.5,4.4) {{\LARGE *}};
\node at (2.5,3.4) {{\LARGE *}};
\node at (2.5,2.4) {{\LARGE *}};
\node at (2.5,1.4) {{\LARGE *}};
\node at (2.5,0.4) {{\LARGE *}};
\node at (3.5,2.4) {{\LARGE *}};
\node at (3.5,1.4) {{\LARGE *}};
\node at (3.5,0.4) {{\LARGE *}};
\node at (4.5,2.4) {{\LARGE *}};
\node at (4.5,1.4) {{\LARGE *}};
\node at (5.5,2.4) {{\LARGE *}};
\node at (5.5,1.4) {{\LARGE *}};
\node at (4.5,0.4) {{\LARGE *}};
\node at (5.5,0.4) {{\LARGE *}};

\node at (.5,6.4) {1};
\node at (2,5.4) {2};
\node at (4.5,3.4) {3};
\node at (-0.5,4.9) {2};
\node at (1.5,1.9) {4};

\end{tikzpicture}\hspace{1.5cm}\begin{tikzpicture}[scale=1.3]
	\def\Node{\node [circle,  fill, inner sep=2pt]}
	\node at (0,0) {};
    \Node[label=left:$v_1$] (1) at (0,1.8) {};
	\Node[label=left:$v_2$] (2) at (1,1.8) {};
	\Node[label=left:$v_3$] (3) at (2,1.8) {};
	\Node[label=left:$v_4$] (4) at (3,1.8) {};
	\Node[label=left:$v_5$] (5) at (4,1.8) {};
	\Node[label=left:$v_6$] (6) at (5,1.8) {};
	\draw (1) to[bend left=50] (2);
	\draw (3) to[bend left=50] (6);
	\draw (4) to[bend left=50] (5);
	\draw (2) to[bend right=50] (3);
	\draw (4) to[bend right=50] (6);
\end{tikzpicture}$$
\caption{A seaweed of type $\frac{2|4}{1|2|3}$ and its associated meander}
\label{fig:seaweed}
\end{figure}
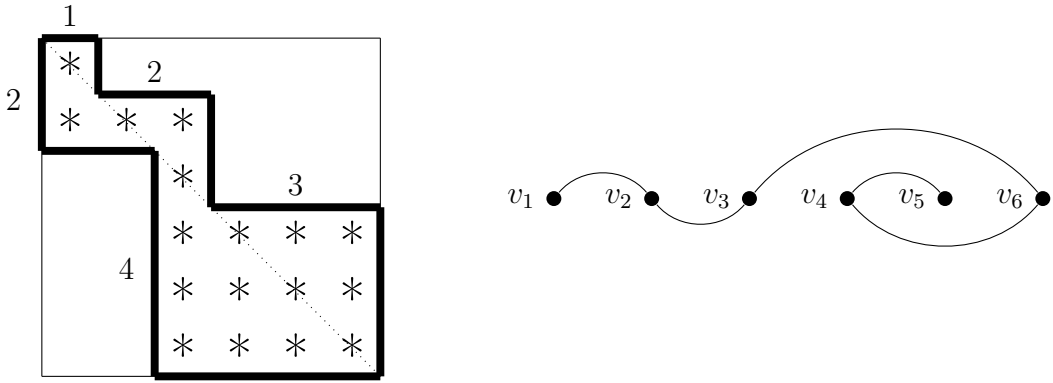

The \textit{index} of a Lie algebra  was introduced by Dixmier  \textbf{\cite{Dix}}.   Formally, the index of a Lie algebra $\mathfrak{g}$ is defined by 

\[\ind (\mathfrak{g})=\min_{f\in \mathfrak{g^*}} \dim  (\ker (B_f)),\]

\noindent where $f$ is a linear form on $\mathfrak{g}$ and $B_f$ is the associated skew-symmetric \textit{Kirillov form} defined by $B_f(x,y)=f([x,y])$ for all $x,y\in\mathfrak{g}$.  The index is an important algebraic invariant of the Lie algebra -- though notoriously difficult to compute.  However, in \textbf{\cite{dk}}, Dergachev and A. Kirillov developed a combinatorial algorithm to compute the index of a seaweed subalgebra of $\mathfrak{sl}(n)$ by counting the number of  connected components of a certain planar graph, called a meander, associated to the seaweed.  To construct a meander, let  
$\frac{ a_1|\cdots|a_m}{b_1|\cdots|b_t}$ be a seaweed.  Now 
label the $n$ vertices of our meander as $v_1, v_2, \dots , v_n$ from left to right along a horizontal line. We then place edges above the horizontal line, called top edges, according to $a_1+\hdots+ a_m$ as follows. 
Partition the set of vertices into a set partition by grouping together the first $a_1$ vertices, then the next $a_2$ vertices, and so on, lastly grouping together the final $a_m$ vertices. We call each set within a set partition a \textit{block}. For each block in the set partition determined by $a_1+\hdots + a_m$, add an edge from the first vertex of the block to the last vertex of the block, then add an edge between the second vertex of the block and the second to last vertex of the block, and so on within each block. More explicitly, given vertices $v_j,v_k$ in a block of 
size $a_i$, there is an edge between them if and only if 
$j+k=2(a_1+a_2+\dots+a_{i-1})+a_i+1$.
In the same way, place bottom edges below the horizontal line of vertices according to the blocks in the partition determined by $b_1+\hdots + b_t$. See the right side of Figure \ref{fig:seaweed}.

Every meander consists of a disjoint union of cycles and paths.  The main result of \textbf{\cite{dk}} is that the index of a seaweed can be computed by counting the number and type of these components in it's associated meander. 

\begin{theorem}\label{thm:dk}  \rm{(Dergachev and A. Kirillov, \textbf{\cite{dk}})} 
If $\mathfrak{p}$ is a seaweed subalgebra of $\mathfrak{sl}(n)$, then
$$\ind (\mathfrak{p}) =2C + P -1,$$
where $C$ is the number of cycles and $P$ is the number of paths in the associated meander.
\end{theorem}

\noindent

\noindent
\begin{example} 
In the example of Figure 1, the meander associated to the seaweed $\frac{2|4}{1|2|3}$ has no cycles and consists of a single path -- so, has index zero, hence is Frobenius.  
\end{example}

\bigskip
While Theorem \ref{thm:dk} is an elegant combinatorial result it is difficult to apply in practice.  However, Coll et al in \textbf{\cite{Collar}} show that any meander can be contracted or ``wound-down" to the empty meander through a sequence of graph-theoretic moves, each of which is uniquely determined by the structure of the meander at the time of move application.  There are five such moves, only one of which affects the component structure of the meander graph and is therefore the only move capable of modifying the index of the meander.  Using these winding-down moves the authors in \textbf{\cite{Collar}} established the following index formulas which allow us to ascertain the index directly from the block sizes of the flags that define the seaweed.\footnote{A recent result by Karnauhova and Liebscher \textbf{\cite{Kar}} has established, in particular, that the formulas in Theorems \ref{2parts} and \ref{3parts} are the only nontrivial linear ones that are available in the maximal parabolic case. More specifically, If $m\geq 4$ and $\mathfrak{p}$ is a seaweed of type $\dfrac{a_1|a_2|\cdots|a_m}{n}$, then there do not exist homogeneous polynomials $f_1,f_2 \in \mathbb{Z}[x_1,...,x_m]$, of arbitrary degree, such that the index of $\mathfrak{p}$ is given by $\gcd (f_1(a_1,...,a_m),f_2(a_1,...,a_m))$.}

\begin{theorem}[Theorem 7, \textbf{\cite{Coll2}}]\label{2parts}
A seaweed of type $\dfrac{a|b}{n}$ has index $\gcd (a,b)-1.$
\end{theorem}

\begin{theorem}[Theorem 8,\textbf{ \cite{Coll2}}]\label{3parts}
A seaweed of type $\dfrac{a|b|c}{n}$, or of type $\dfrac{a|b}{c|n-c}$, has index given by $\gcd(a+b,b+c)-1.$
\end{theorem}

Since we will need the explicit winding-down moves in the proof of Theorem~\ref{thm:period} we review the winding-down process.

\begin{lemma}[Winding-down]\label{lem:wd} Given a meander $M$ of type $\dfrac{a_1|a_2|...|a_m}{b_1|b_2|...|b_t}$, create a meander $M'$
by exactly one of the following moves. For all moves except the 
Component Elimination move, $M$ and $M'$ have the same index. 
\begin{enumerate}
\item {\bf Vertical Flip $(F_v)$:} If $a_1<b_1$, then 
$M'$ has type 
$\displaystyle\frac{b_1|b_2|...|b_t}{a_1|a_2|...|a_m}$.
    
\item {\bf Component Elimination $(C(c))$:} 
If $a_1=b_1=c$, then $M'$ has type 
$\displaystyle\frac{a_2|a_3|...|a_m}{b_2|b_3|...|b_t}.$

\item {\bf Rotation Contraction $(R)$:} If $b_1<a_1<2b_1$, then 
$M'$ has type 
$\displaystyle \frac{b_1|a_2|a_3|...|a_m}{(2b_1-a_1)|b_2|...|b_t}$.

\item {\bf Block Elimination $(B)$:} If $a_1=2b_1$, then 
$M'$ has type 
$\displaystyle\frac{b_1|a_2|..|a_m}{b_2|b_3|...|b_t}$.

\item {\bf Pure Contraction $(P)$:} If $a_1>2b_1$, then 
$M'$ has type 
$\displaystyle
\frac{(a_1-2b_1)|b_1|a_2|a_3|...|a_m}{b_2|b_3|...|b_t}$.

\end{enumerate}
\end{lemma}

\noindent
\begin{example} In this example, the seaweed $\frac{17|3}{10|4|6}$ is wound-down to the empty meander using the moves detailed in Lemma~\ref{lem:wd}.

\begin{figure}[H]
$$\begin{tikzpicture}[scale=.4]
\def\Node{\node [circle, fill, inner sep=2pt]}
\node at (10.5,-2){$\frac{17|3}{10|4|6}$};
\Node (1) at (1,0){};
\Node (2) at (2,0){};
\Node (3) at (3,0){};
\Node (4) at (4,0){};
\Node (5) at (5,0){};
\Node (6) at (6,0){};
\Node (7) at (7,0){};
\Node (8) at (8,0){};
\Node (9) at (9,0){};
\Node (10) at (10,0){};
\Node (11) at (11,0){};
\Node (12) at (12,0){};
\Node (13) at (13,0){};
\Node (14) at (14,0){};
\Node (15) at (15,0){};
\Node (16) at (16,0){};
\Node (17) at (17,0){};
\Node (18) at (18,0){};
\Node (19) at (19,0){};
\Node (20) at (20,0){};
\draw (1) to[bend left] (17);
\draw (2) to[bend left] (16);
\draw (3) to[bend left] (15);
\draw (4) to[bend left] (14);
\draw (5) to[bend left] (13);
\draw (6) to[bend left] (12);
\draw (7) to[bend left] (11);
\draw (8) to[bend left] (10);
\draw (18) to[bend left] (20);
\draw (1) to[bend right] (10);
\draw (2) to[bend right] (9);
\draw (3) to[bend right] (8);
\draw (4) to[bend right] (7);
\draw (5) to[bend right] (6);
\draw (11) to[bend right] (14);
\draw (12) to[bend right] (13);
\draw (15) to[bend right] (20);
\draw (16) to[bend right] (19);
\draw (17) to[bend right] (18);
\end{tikzpicture}
\hspace{1em}
\begin{tikzpicture}[scale=.4]
\def\Node{\node [circle, fill, inner sep=2pt]}
\node at (-2,0){$\overset{R}{\mapsto}$};
\node at (6,-2){$\frac{10|3}{3|4|6}$};
\Node (1) at (1,0){};
\Node (2) at (2,0){};
\Node (3) at (3,0){};
\Node (4) at (4,0){};
\Node (5) at (5,0){};
\Node (6) at (6,0){};
\Node (7) at (7,0){};
\Node (8) at (8,0){};
\Node (9) at (9,0){};
\Node (10) at (10,0){};
\Node (11) at (11,0){};
\Node (12) at (12,0){};
\Node (13) at (13,0){};
\draw (1) to[bend left] (10);
\draw (2) to[bend left] (9);
\draw (3) to[bend left](8);
\draw (4) to[bend left](7);
\draw (5) to[bend left](6);
\draw (11) to[bend left](13);
\draw (1) to[bend right](3);
\draw (4) to[bend right](7);
\draw (5) to[bend right](6);
\draw (8) to[bend right](13);
\draw (9) to[bend right](12);
\draw (10) to[bend right](11);
\end{tikzpicture}$$
$$\begin{tikzpicture}[scale=.4]
\def\Node{\node [circle, fill, inner sep=2pt]}
\node at (-1,0){$\overset{P}{\mapsto}$};
\node at (6,-2){$\frac{4|3|3}{4|6}$};
\Node (1) at (1,0){};
\Node (2) at (2,0){};
\Node (3) at (3,0){};
\Node (4) at (4,0){};
\Node (5) at (5,0){};
\Node (6) at (6,0){};
\Node (7) at (7,0){};
\Node (8) at (8,0){};
\Node (9) at (9,0){};
\Node (10) at (10,0){};
\draw (1) to[bend left] (4);
\draw (2) to[bend left](3);
\draw (5) to[bend left](7);
\draw (8) to[bend left](10);
\draw (1) to[bend right](4);
\draw (2) to[bend right](3);
\draw (5) to[bend right](10);
\draw (6) to[bend right](9);
\draw (7) to[bend right](8);
\end{tikzpicture}
\hspace{1em}
\begin{tikzpicture}[scale=.4]
\def\Node{\node [circle, fill, inner sep=2pt]}
\node at (-1,0){$\overset{C(4)}{\mapsto}$};
\node at (3.5,-2){$\frac{3|3}{6}$};
\Node (1) at (1,0){};
\Node (2) at (2,0){};
\Node (3) at (3,0){};
\Node (4) at (4,0){};
\Node (5) at (5,0){};
\Node (6) at (6,0){};
\draw (1) to[bend left](3);
\draw (4) to[bend left](6);
\draw (1) to[bend right](6);
\draw (2) to[bend right](5);
\draw (3) to[bend right](4);
\end{tikzpicture}
\hspace{1em}
\begin{tikzpicture}[scale=.4]
\def\Node{\node [circle, fill, inner sep=2pt]}
\node at (-1,0){$\overset{F_v}{\mapsto}$};
\node at (3.5,-2){$\frac{6}{3|3}$};
\Node (1) at (1,0){};
\Node (2) at (2,0){};
\Node (3) at (3,0){};
\Node (4) at (4,0){};
\Node (5) at (5,0){};
\Node (6) at (6,0){};
\draw (1) to[bend right](3);
\draw (4) to[bend right](6);
\draw (1) to[bend left](6);
\draw (2) to[bend left](5);
\draw (3) to[bend left](4);
\end{tikzpicture}
\hspace{1em}
\begin{tikzpicture}[scale=.4]
\def\Node{\node [circle, fill, inner sep=2pt]}
\node at (-1,0){$\overset{B}{\mapsto}$};
\node at (2,-2){$\frac{3}{3}$};
\Node (1) at (1,0){};
\Node (2) at (2,0){};
\Node (3) at (3,0){};
\draw (1) to[bend right](3);
\draw (1) to[bend left](3);
\end{tikzpicture}
\hspace{1em}
\begin{tikzpicture}[scale=.4]
\node at (-1,0){$\overset{C(3)}{\mapsto}$};
\node at (1,-2){$\emptyset$};
\node at (1,0){$\emptyset$};
\end{tikzpicture}$$
\caption{
Winding down the meander $\dfrac{17|3}{10|4|6}$}
\label{fig:signature}
\end{figure}
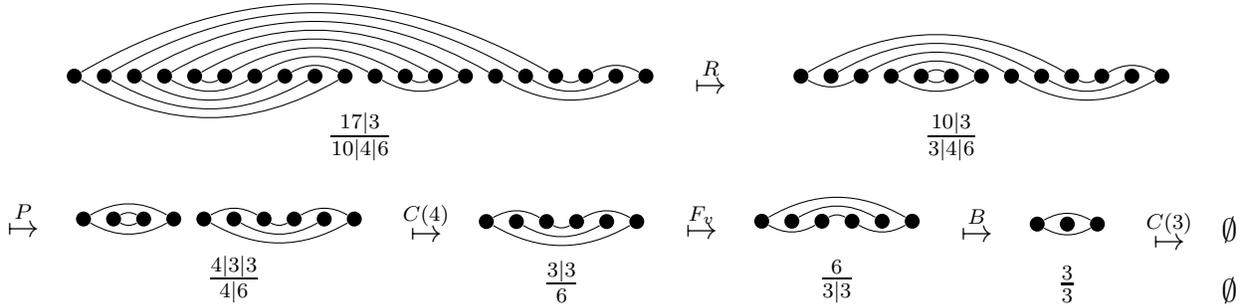
\end{example}

\noindent
In what follows, it is helpful to add a sixth 
index preserving transformation, $F_h$, called a \textit{horizontal flip} which takes $M$ to $\dfrac{a_m|...|a_2|a_1}{b_m|...|b_2|b_1}$.

\section{The index of a partition}\label{sect:results}
Let $\mathcal{P}(n)$ be the set of integer partitions of a positive integer $n$ and let $\lambda, \mu \in \mathcal{P}(n)$ with 
$\lambda = (\lambda_1, \lambda_2, \dots, \lambda_m)$ and 
$\mu = (\mu_1, \mu_2, \dots, \mu_t).$  
These compositions can be used to define the seaweed 
$$\mathfrak{p}(\lambda, \mu)=\frac{\lambda_1| \lambda_2| \dots| \lambda_m}{\mu_1| \mu_2| \dots| \mu_t}.$$
We can then define the \textit{index} of the pair $(\lambda, \mu)$ to be the index of $\mathfrak{p}(\lambda, \mu)$, and we write ind$_{\mu}(\lambda)$.  Given $\lambda$ as above, there are two choices for $\mu$ naturally associated with $\lambda$, namely, $w(\lambda)$ and its conjugate $1^{w(\lambda)}$. These yield, respectively, two partition statistics on $\lambda$ defined as
follows:

\begin{eqnarray}\label{index}
\textrm{ind}_{w(\lambda)}(\lambda) =\textrm{ind}\left(\frac{\lambda_1|\hdots|\lambda_m}{w(\lambda)}\right)~~~\textrm{and}~~~ 
\textrm{ind}_{1^{w(\lambda)}}(\lambda)=\textrm{ind}\left(\frac{\lambda_1|\hdots|\lambda_m}{1|\hdots|1}\right).
\end{eqnarray}

In the first seaweed, the bottom composition is defined by the trivial partition, yielding a maximal parabolic seaweed.  In the second case,
the bottom composition consists of $w(\lambda)$ $1's$.

\begin{example} Let $\lambda=(3,2,1)$ in (\ref{index}).  An application of Theorem 3 now yields

$$
\textrm{ind}_6(\lambda) =\textrm{ind}\left(\frac{3|2|1}{6}\right)=0~~~\textrm{and}~~~ \textrm{ind}_{1^6}(\lambda)=\textrm{ind}\left(\frac{3|2|1}{1|1|1|1|1|1}\right)=3.
$$
\end{example}


\subsection{All 1's}\label{sect:1s}

In this section we investigate, for fixed $i$ and varying $n$, the sequence of values defined by the number of partitions $\lambda\in\mathcal{P}(n)$ such that $\ind_{1^{w(\lambda)}}(\lambda)=n-i$. We find that for each $i$, if $c^i_n=|\{\lambda\in\mathcal{P}(n)|\ind_{1^{w(\lambda)}}(\lambda)=n-i\}|$, then $\{c_n^i\}_{n=i}^{\infty}$ is eventually constant, converging to a well-known classical value $c^i$ (see Theorem~\ref{thm:ones}). The following Table 1, illustrates $c_n^i$ for small values of $n$ and $i$.

\begin{table}[H]
\centering
\begin{tabular}{|c|c|c|c|c|c|c|c|c|c|c|}
\hline
$\boldsymbol{n}\backslash\boldsymbol{i}$ & \textbf{0}  & \textbf{1}  & \textbf{2}  & \textbf{3} & \textbf{4} & \textbf{5} & \textbf{6} & \textbf{7} & \textbf{8} & \textbf{9}  \\ \hline
\textbf{1}        & 1  & 0  & 0  & 0 & 0 & 0 & 0 & 0 & 0 & 0   \\ \hline
\textbf{2}        & 1  & 1  & 0  & 0 & 0 & 0 & 0 & 0 & 0 & 0  \\ \hline
\textbf{3}        & 0  & 2  & 1  & 0 & 0 & 0 & 0 & 0 & 0 & 0  \\ \hline
\textbf{4}        & 0 & 2  & 2  & 1 & 0 & 0 & 0 & 0 & 0 & 0   \\ \hline
\textbf{5}        & 0  & 0  & 4  & 2 & 1 & 0 & 0 & 0 & 0 & 0   \\ \hline
\textbf{6}       & 0  & 0  & 3  & 5 & 2 & 1 & 0 & 0 & 0 & 0   \\ \hline
\textbf{7}        & 0 & 0 & 0  & 7 & 5 & 2 & 1 & 0 & 0 & 0   \\ \hline
\textbf{8}        & 0 & 0  & 0 & 5 & 9 & 5 & 2 & 1 & 0 & 0   \\ \hline
\textbf{9}       & 0 & 0 & 0  & 0 & 12 & 10 & 5 & 2 & 1 & 0   \\ \hline
\textbf{10}       & 0 & 0  & 0  & 0 & 7 & 17 & 10 & 5 & 2 & 1  \\ \hline
\end{tabular}
\caption{Number of $\lambda\in\mathcal{P}(n)$ with $\ind_{1^{w(\lambda)}}(\lambda)=i$.}\label{tab:1s}
\end{table}

By coloring partitions, we can better understand the $c^i$'s. 
We will use two colors (red and blue), to color the parts of a given  partition. When enumerating colored partitions, we will assume that two partitions which are identical, save for their coloring, will be considered different partitions. So, for example the partition of the integer 2 given by $(\bl{1},\bl{1})$ is different from the partition of the integer 2 given by $(\bl{1},\rn{1})$. We also tacitly assume that in a given colored partition all blue parts of a given size precede all red parts of the same size. See Example 9.  

\bigskip
\noindent
\textit{Remark: } In the classical literature such (two)-colored partitions are called partitions into parts of two \textit{kinds}.  Partitions into two kinds can be found in Guptas' (\textbf{\cite{gupta}},1958) and have recently been connected to other objects as diverse as quandles \textbf{\cite{quandle2,quandle1}}. 

\begin{example}
The (two)-colored partitions of 2 are: (\rn{2}), (\bl{2}), (\rn{1},\rn{1}), (\bl{1},\bl{1}), and (\bl{1},\rn{1}).   
\end{example}

\noindent
The generating function for the number of (two)-colored partitions of $n$ is well-known and is equal to 

$$\prod_{m\ge 1}\frac{1}{(1-x^m)^2}.$$


The following theorem connects the current exposition to classical partition theory.

\begin{theorem}\label{thm:ones}
$$\sum_{i\ge 1}c^ix^{i-1}=\prod_{m\ge 1}\frac{1}{(1-x^m)^2}.$$
\end{theorem}
\begin{proof}
The case $i=1$ is clear since, by Theorem~\ref{thm:dk}, the only $\lambda\in\mathcal{P}(n)$ with $\ind_{1^{w(\lambda)}}(\lambda)=n-1$ is $\lambda=1^n$. We show that for $i>1$ and $n\ge 3i-3$, there is a bijective correspondence between $\mathcal{M}(i,n)=\{\lambda\in\mathcal{P}(n)~: ~\ind_{1^{w(\lambda)}}(\lambda)=n-i\}$ and $\mathcal{P}^2(i-1)=\{\text{colored partitions of }i-1\}$. We do this in two steps.

First, let $\mathcal{M}(i-1)$ 
be the set of partitions $\mu=(\mu_1,\hdots,\mu_m)$ such that $\mu_m>1$ and the meander corresponding to $\mathfrak{p}(\mu,1^{w(\mu)})$ has $i-1$ arcs. Consider the map $\varphi$ which takes $\lambda=(\lambda_1,\hdots,\lambda_m)\in \mathcal{P}^2(i-1)$ to the partition $\varphi(\lambda)=(\mu_1,\hdots,\mu_m)$ defined by 

\[\mu_i = \begin{cases} 
     2\lambda_i+1, & \lambda_i\text{ is blue}; \\
     2\lambda_i, & \lambda_i\text{ is red.} 
   \end{cases}
\] 

\noindent
By construction, $\varphi(\lambda)\in \mathcal{M}(i-1)$. Furthermore, $\varphi$ is invertible so $\varphi$ is a bijection between $\mathcal{P}^2(i-1)$ and $\mathcal{M}(i-1)$. Via this correspondence, it is easy to see that the largest partition of $\mathcal{M}(i-1)$ has weight $3i-3$.

 By Theorem~\ref{thm:dk}, $\ind_{1^{w(\lambda)}}(\lambda)$ corresponds to $n$ minus the number of arcs in the meander minus 1. Thus, $\mathcal{M}(i,n)$ consists of partitions $\lambda\in\mathcal{P}(n)$ such that the meander corresponding to $\mathfrak{p}(\lambda,1^{w(\lambda)})$ has exactly $i-1$ arcs. Therefore, if $n\ge 3i-3$, then elements of $\lambda\in \mathcal{M}(i,n)$ can be mapped bijectively to elements of $\mathcal{M}(i-1)$ by the map $\psi$ which removes all parts equal to 1. The required bijection is  given by $\psi \circ \varphi$.  \end{proof}

\subsection{The maximal parabolic case}\label{sect:n}
As above, let $\lambda = (\lambda_1, \dots, \lambda_m)$ be an element of $\mathcal{P}(n)$.
In this section, we consider the seaweed defined by the pair of compositions $(\lambda, w(\lambda))$. In contrast to the previous section, here we investigate the number of partitions $\lambda$ such that $\ind_{w(\lambda)}(\lambda)=0$.
We naturally call such partitions, \textit{Frobenius partitions}. The main theorem of this section, Theorem~\ref{thm:period}, remarkably establishes that if $\lambda_i\leq 7$, for $i=1,\dots,m$, then the number of Frobenius partitions is a periodic function of $n$. 
 
We begin with two Lemmas which will be helpful in the proof of Theorem~\ref{thm:period}.

\begin{lemma}\label{lem:sum}
Let $\mathfrak{g}=\frac{a_1|\hdots|a_m}{\sum_{i=1}^ma_i}$ be a seaweed algebra. If there exists $i<j-1$ such that $\sum_{l=1}^ia_l=\sum_{l=j}^ma_l$, then $\mathfrak{g}$ is not Frobenius.
\end{lemma}
\begin{proof}
Applying the winding moves $(F)$ followed by $i$ applications of $(P)$ to the meander corresponding to $\mathfrak{g}$ results in the meander corresponding to the seaweed algebra of type $\frac{b_1|a_i|\hdots|a_1}{a_{i+1}|\hdots|a_j|\hdots|a_m}$ where $b_1=\sum_{l=i+1}^{j-1}a_l>0$; but this meander consists of at least two components, one corresponding to $\frac{b_1}{a_{i+1}|\hdots|a_{j-1}}$, and the other $\frac{a_i|\hdots|a_1}{a_j|\hdots|a_m}$. Thus, by Theorem~\ref{thm:dk}, $\ind(\mathfrak{g})>0.$
\end{proof}

\begin{lemma}\label{lem:odd}
Let $\mathfrak{g}=\frac{a_1|\hdots|a_m}{\sum_{i=1}^ma_i}$ be a seaweed algebra. If there exists more than two odd $a_i$'s, then $\mathfrak{g}$ is not Frobenius.
\end{lemma}
\begin{proof}
Each odd $a_i$ contributes a vertex of degree 1 to the meander corresponding to $\mathfrak{g}$. Recall that each open path consists of exactly two vertices of degree 1 and no closed paths contains a vertex of degree 1. So, if there are more than two odd $a_i$'s, then the corresponding meander must contain more than one open path and thus, by Theorem~\ref{thm:dk}, $\ind(\mathfrak{g})>0.$
\end{proof}

Let $\mathcal{P}(n,d)$ be the set of Frobenius partitions $\lambda=(\lambda_1,\hdots,\lambda_m)\in\mathcal{P}(n)$ such that $\lambda_i\le d$ for $1\le i\le m$.

\begin{theorem}\label{thm:period}
If $d\in\{1,2,3,4\}$, then the values of $|\mathcal{P}(n,d)|$ are eventually periodic. More precisely, 
\begin{itemize}
	\item If $n\ge 3$, $|\mathcal{P}(n,1)|=0$
    \item If $n\ge 5$, \[|\mathcal{P}(n,2)| = \begin{cases} 
     1, & n\text{ odd} \\
     0, & n\text{ even}
   \end{cases}
\] 
	\item If $n\ge 13$, \[|\mathcal{P}(n,3)| = \begin{cases} 
     2, & n\text{ odd} \\
     0, & n\text{ even}
   \end{cases}
\] 
	\item If $n\ge 17$, \[|\mathcal{P}(n,4)| = \begin{cases} 
     4, & n\equiv 1(\text{mod }4) \\
     2, & n\equiv 2(\text{mod }4) \\
     3, & n\equiv 3(\text{mod }4) \\
     0, & n\equiv 0(\text{mod }4) \\
   \end{cases}
\] 
\end{itemize}
\end{theorem}
\begin{proof} The proof heuristic is described as follows. We consider the possible partitions for each $d\le 4$ -- except for those cases considered in Lemma~\ref{lem:sum} and Lemma~\ref{lem:odd} -- in reverse lexicographic ordering and determine which partitions are Frobenius.

$\mathbf{d=1}$: $\ind(\frac{1|\hdots|1}{n})=\lfloor\frac{n}{2}\rfloor>0$ for $n\ge 3$. Thus, for $n\ge 3$ we have $|\mathcal{P}(n,1)|=0$.

$\mathbf{d=2,n\ge 5}$: Using the results determined for the case $d=1$, we consider only partitions with $\lambda_1=2$. After applying Lemma~\ref{lem:sum} and Lemma~\ref{lem:odd}, the only partitions remaining are those of the form $(1^12^{f_2})$, which are Frobenius by Theorem 10 of \textbf{\cite{Coll2}}. Thus, there is exactly one such Frobenius partition if and only if the weight is odd.

$\mathbf{d=3,n\ge 13}$: As before, using the results for the cases $d=1$ and $d=2$, we can restrict our attention to partitions with $\lambda_1=3$. After applying Lemma~\ref{lem:sum} and Lemma~\ref{lem:odd}, the only partitions remaining are those of the form $(2^{f_2}3^1)$, which are Frobenius, once again, by Theorem 10 of \textbf{\cite{Coll2}}. Thus, as in the case $d=2$, there is exactly one such Frobenius partition if and only if the weight is odd.

$\mathbf{d=4,n\ge 17}$: Finally, using the results for $d=1,2,3$, we consider only partitions with $\lambda_1=4$. After applying Lemma~\ref{lem:sum} and Lemma~\ref{lem:odd}, we are left with seven cases to consider.
\begin{enumerate}
    \item $(3^14^{f_4})$: Partitions of this form are Frobenius by Theorem 10 of \textbf{\cite{Coll2}}.
    \item $(3^24^{f_4})$: Applying the sequence of moves $(F_v),(P),(F_h),(R),(R),(B),(F_v),(P),(F_h)$ to the corresponding meander results in the meander for a partition of the form $(1^1f^{f_4})$, which is found to be Frobenius in case 5 below.
    \item $(2^14^{f_4})$: Applying the sequence of moves $(F_h),(F_v),(P),(F_v),(F_h),(B)$ inductively to the corresponding meander results in the meander for the seaweed algebra of type $\frac{2}{2}$, which has index 1.
    \item $(2^13^14^{f_4})$: Applying the sequence of moves $(F_v),(P),(F_h),(B),(F_v),(R),(B),(F_h)$ to the corresponding meander results in the meander for a partition of the form $(1^14^{f_4})$, which is found to be Frobenius in case 5 below.
    \item $(1^14^{f_4})$: Partitions of this form are Frobenius by Theorem 10 of \textbf{\cite{Coll2}}.
    \item $(1^12^{f_2}4^{f_4}),f_2\ge 1$: Applying the sequence of moves $(F_v),(P),(F_h),(P)$ to the corresponding meander results in a seaweed algebra of type $\frac{2|1|n-8}{2|\hdots|2|4|\hdots|4}$, which splits into at least two components, $\frac{2}{2}$ and $\frac{1|n-8}{2|\hdots|2|4|\hdots|4}$, and is therefore not Frobenius by Theorem~\ref{thm:dk}.
    \item $(1^24^{f_4})$: Applying the sequence of moves $(F_h)(F_v)(P)(P)(F_v)(F_h)(P)(B)(F_h)$ to the corresponding meander results in the meander for a partition of the form $(1^14^{f_4})$, which was found to be Frobenius in case 5 above.
\end{enumerate}
Thus, there are two Frobenius partitions with weight congruent to 1(mod 4); two with weight congruent to 2(mod 4); and one with weight congruent to 3(mod 4).
\end{proof}

\noindent
\textit{Remark:} Similar methods to those used in the proof of Theorem~\ref{thm:period} can be used to establish periodic behavior for $d\in\{5,6,7\}$. In the case of $d=5$, the period is of length 4 -- with values 7,3,5,3 -- while in the cases of $d\in\{6,7\}$, the period jumps to 14.  See Example~\ref{ex:period}.

\begin{example}\label{ex:period}
The sequence of values of $|\mathcal{P}(n,d)|$ for $d\in\{5,6,7\}$, along with the value of $n$ at which $|\mathcal{P}(n,d)|$ becomes periodic. 
$$d=5,n\ge 21:\text{ }7,3,5,3$$
$$d=6,n\ge 37:\text{ }14,5,9,3,11,5,11,3,12,5,8,3$$
$$d=7,n\ge 41:\text{ }19,9,18,7,19,9,17,7,20,9,17,7$$
\end{example}

\noindent
\textit{Remark:}
At $d=8$ the periodicity stops, which can be seen by considering $|\mathcal{P}(n,8)|$ for $n\equiv 1(\text{mod }8)$ where for $n=8m+1$ we have that $\ind_{w(\lambda)}(1^14^{2(m-k)}8^k)=0$.
\\*


\section{Conclusion}
Using the index theory of seaweed algebras we advance the notion of the index of a partition pair, by simply defining the index of the latter to be the index of the former.  This rather pedestrian definition allows us to describe various statistics on integer partitions.  We consider the two extremal cases defined by $(\lambda, 1^{w(\lambda)})$ and $(\lambda, w(\lambda))$.  But what about other $\lambda$-based choices for the second composition?    
 For example one could pair a partition $\lambda=(\lambda_1,\hdots,\lambda_m)$ with its reverse $\text{Rev}(\lambda)=(\lambda_m,\hdots,\lambda_1)$. It follows from a result of \textbf{\cite{collcomp}} that $$|\{\lambda\in \mathcal{P}(n)~:~\ind_{\text{Rev}(\lambda)}(\lambda)=n-1\}|=d(n).$$ Alternatively, a partition can be paired with its conjugate. In this case, via the same result of \textbf{\cite{collcomp}}, we find that $|\{\lambda\in \mathcal{P}(n)~:~\ind_{\lambda^C}(\lambda)=n-1\}|$ is equal to twice the number of self-conjugate partitions of $n$. 
 
We might also consider incorporating weighted sums, such as those that appear in Euler's Pentagonal Number Theorem and other Legendre type theorems \textbf{\cite{legendre}}. In such results, partitions $\lambda$ contribute a term of the form $(-1)^{l(\lambda)}q^{w(\lambda)}$ to the weighted sum. One could instead insist that each partition contributes a term of the form $(-1)^{\ind_{w(\lambda)}(\lambda)}q^{w(\lambda)}$. For example, by restricting to partitions with only odd parts (denoted $\mathcal{P}(n,S_{odd})$) and considering the sets $$e_n=|\{\lambda\in \mathcal{P}(n,S_{odd})~:~\ind_{w(\lambda)}(\lambda)\text{ is even}\}|\quad o_n=|\{\lambda\in \mathcal{P}(n,S_{odd})~:~\ind_{w(\lambda)}(\lambda)\text{ is odd}\}|,$$  numerical data suggests the following interesting conjecture.

\begin{conjecture}\label{conj:ws}

\begin{eqnarray}\label{conjecture}
\sum_{n\ge 0}|e_n-o_n|q^n=\prod_{k\ge 1}\frac{1}{1+(-1)^kq^{2k-1}}.
\end{eqnarray}
\end{conjecture}

\noindent
The sequence of coefficients of the product in (\ref{conjecture}) can be found on the Online Encyclopedia of Integer Sequences (``OEIS") as A300574, where this sequence is further conjectured to be nonnegative. If true, 
 Conjecture~\ref{conj:ws} would not only establish that the sequence A300574 is nonnegative but (\ref{conjecture}) would also provide a combinatorial interpretation.


\bibliographystyle{abbrv}

\bibliography{bibliography_partitions.bib}

\end{document}